\documentclass[notitlepage]{amsart}
\usepackage{amsfonts}
\usepackage[latin1]{inputenc}
\usepackage{amsthm}
\usepackage{amsmath}
\usepackage{xypic}
\usepackage{mathtools}
\usepackage{amsfonts}

\usepackage{tikz-cd}

\newtheorem{pro}{Proposition}[section]
\newtheorem{teo}[pro]{Theorem}
\newtheorem{defi}[pro]{Definition}
\newtheorem{lem}[pro]{Lemma}
\newtheorem{cor}[pro]{Corollary}
\newtheorem{remark}[pro]{Remark}
\newtheorem{ex}[pro]{Example}

\newcommand{\pd}{{\mathrm{pd}}}
\newcommand{\Gpd}{{\mathrm{Gpd}}}

\newcommand{\modu}{{\mathrm{mod}}}

\newcommand{\ind}{{\mathrm{ind}}}

\newcommand{\Ker}{{\mathrm{Ker}}}

\newcommand{\Gproj}{{\mathrm{Gproj}}}

\newcommand{\id}{{\mathrm{id}}}

\newcommand{\findim}{{\mathrm{fin.dim}}}

\newcommand{\resdim}{{\mathrm{resdim}}}

\newcommand{\coresdim}{{\mathrm{coresdim}}}

\newcommand{\Hom}{{\mathrm{Hom}}}

\newcommand{\End}{{\mathrm{End}}}

\newcommand{\Ext}{{\mathrm{Ext}}}
\newcommand{\Phidim}{{\Phi\,\mathrm{dim}}}
\newcommand{\Psidim}{{\Psi\,\mathrm{dim}}}
\newcommand{\add}{{\mathrm{add}}}

\newcommand{\D}{\mathcal{D}}
\newcommand{\X}{\mathcal{X}}
\newcommand{\Y}{\mathcal{Y}}

\newcommand{\p}{\mathcal{P}}

\newcommand{\Z}{\mathbb{Z}}

\title{Generalised Igusa-Todorov functions and Lat-Igusa-Todorov algebras}
\author{Diego Bravo$^{1}$, Marcelo Lanzilotta$^{1}$, Octavio Mendoza$^{2}$\\ and Jos\'e Vivero$^{1}$}
\thanks{$^{1}$ GIA - Grupo de Investigaci\'on en \'Algebra (CSIC 323725), Instituto de Matem\'atica y Estad\'istica ``Rafael Laguardia'', Facultad de Ingenier\'\i a, Universidad de la Rep\'ublica, Uruguay.}
\thanks{$^{2}$ Instituto de Matem\'aticas, Universidad Nacional Aut\'onoma de M\'exico, M\'exico.}
\thanks{The third author thanks the project PAPIIT-Universidad Nacional Aut\'onoma de M\'exico IN100520.}

\numberwithin{equation}{section}

\begin{document}
\maketitle


\begin{center}
\small{dbravo@fing.edu.uy, 
marclan@fing.edu.uy, 
omendoza@matem.unam.mx, 
jvivero@fing.edu.uy} 
\end{center}

\begin{center}
\end{center}

\begin{abstract}
In this paper we study a generalisation of the Igusa-Todorov functions which gives rise to a vast class of algebras satisfying the finitistic dimension conjecture. This class of algebras is called Lat-Igusa-Todorov and includes, among others, the Igusa-Todorov algebras (defined by J. Wei) and the self-injective algebras which in general are not Igusa-Todorov algebras. Finally, some applications of the developed theory are given in order to relate the different homological dimensions which have been discussed through the paper.
\\


\end{abstract}

\section{Introduction}
The aim of this work is to study the homological properties of the category $\modu\,(\Lambda)$ formed by all the finitely generated left $\Lambda$-modules over an Artin algebra $\Lambda.$  For instance, the finitistic dimension of the Artin  algebra $\Lambda$ is defined by $\findim\,(\Lambda):=\sup\{\pd\,M\;:\;\ M\in \modu\,(\Lambda) \;\text{and}\;\pd\,M<\infty\}$, where $\pd\,M$ is the projective dimension of $M.$ This definition gave rise to a famous conjecture, which is the main reason for this article, stating that for any Artin algebra $\Lambda$ there is a bound for the finitistic dimension $\findim\,(\Lambda).$ For further reading into the history of the fintistic dimension conjecture, we recommend \cite{Z}.

In \cite{W}, J. Wei gave the definition of  Igusa-Todorov algebra, by using the functions defined by K. Igusa and G. Todorov in \cite{IT}, and proved that these algebras satisfy the finitistic dimension conjecture. He also provided an extensive list of algebras that turned out to be Igusa-Todorov algebras, conjecturing that all Artin algebras were Igusa-Todorov.  But later on, it was proved that, within the class of self-injective algebras, there are examples of algebras that do not satisfy Wei's definition. More details and references for this are given in Section 4. One of the objectives of this work is to fix this gap by finding an alternative through a more general definition. In order to do so, we came up with a generalisation of the Igusa-Todorov functions, providing us with an appropriate setting to achieve our purpose. This gave rise to a new class of algebras named Lat-Igusa-Todorov algebras. This class of algebras also satisfies the finitistic dimension conjecture and strictly includes the class of algebras defined by Wei, since self-injective algebras verify this new definition.

The technique we have used to generalise the Igusa-Todorov functions is independent of the specific objectives of this work and we think it can be of great interest, not only to the finitistic dimension conjecture, but for other problems in representation theory. 
\

The content of this paper can be summarised as follows. In Section 2, it is collected some necessary material for the developing of this work. In order to clarify our exposition in this section, we bring out to the light the notion of relative resolution dimension. Namely, for an Artin algebra $\Lambda,$ let 
$\X\subseteq\modu\,(\Lambda)$ and $M\in\modu\,(\Lambda).$ The $\X$-resolution dimension $\resdim_\X(M)$ of $M$ is the minimal non-negative integer $n$ such that there is an exact sequence $0\to X_n\to\cdots \to X_1\to X_0\to M\to 0$ with $X_i\in\X$ for all $i.$ If such $n$ does not exist, 
we set $\resdim_\X(M):=\infty.$ We also consider the class $\X^\wedge:=\{M\in\modu\,(\Lambda)\;:\; \resdim_\X(M)<\infty\}.$
\

In Section 3, it is given the proposed generalisation of the Igusa-Todorov functions $\Phi,\Psi:\modu\,(\Lambda)\to \mathbb{N},$ for an Artin algebra 
$\Lambda,$ which were firstly introduced in \cite{IT}. In order to do this, let $K$ be the free abelian group generated by the set of all the iso-classes 
$[M],$ for $M\in \modu\,(\Lambda),$ modulo the relation $[M]-[X]-[Y]$ if $M\simeq X\oplus Y.$ For a given subclass $\X\subseteq \modu\,(\Lambda)$ 
such that $\add\,\X=\X,$ we denote by $\langle\X\rangle$ the free abelian subgroup of $K$ generated by the set of all the iso-classes of indecomposable 
 objects in $\X\cup \p,$ where $\p$ is the class of all the projective objects in $\modu\,(\Lambda).$ In the case of a single object 
 $X\in\modu\,(\Lambda),$ we set $\langle X\rangle:=\langle\add(X)\rangle.$\\
Let us fix a subclass $\D\subseteq \modu\,(\Lambda)$ such that $\add\,\D=\D$ and $\Omega(\D)\subseteq \D.$ Hence, the quotient group $K_\D:=K/\langle\D\rangle$ is  free abelian. Note that $K_\D$ is  the free abelian group generated by the set of all the iso-classes of indecomposable non-projective $\Lambda$-modules that do not lie in $\D.$ Moreover,  the syzygy operator $\Omega$ on $\modu\,(\Lambda)$ induces a morphism 
$\overline{L}:K_\D\to K_\D,$ $[M]+\langle \D\rangle\mapsto [\Omega M]+\langle \D\rangle,$ of abelian groups.\\
For $X\in\modu\,(\Lambda),$ we consider the abelian subgroup $\overline{\langle X\rangle}$ of $K_\D$ given by $\overline{\langle X\rangle}:=(\langle X\rangle+\langle\D\rangle)/\langle\D\rangle.$ By  Fitting's Lemma, there exists a minimum non negative integer, denoted by $\Phi_{[\D]}(X),$ such that 
$\overline{L}:\overline{L}^m(\overline{\langle X\rangle})\to :\overline{L}^{m+1}(\overline{\langle X\rangle})$ is an isomorphism for any 
$m\geq \Phi_{[\D]}(X).$ Thus, we get a function $\Phi_{[\D]}:\modu\,(\Lambda)\to \mathbb{N}.$ Note that, for the case $\D=\p,$ we have that 
$\Phi_{[\D]}=\Phi,$ where $\Phi$ is the usual Igusa-Todorov function introduced in \cite{IT}.\\
Once we have defined the function $\Phi_{[\D]},$ we get the function $\Psi_{[\D]}:\modu\,(\Lambda)\to \mathbb{N}$ defined as follows
\begin{center}
$\Psi_{[\D]}(X):=\Phi_{[\D]}(X) + \findim\,(\add\,\Omega^{\Phi_{[\D]}(X)}(X)),$
\end{center}
where $\findim(\mathcal{Z}):=\sup\{\pd(Z)\;:\; Z\in\mathcal{Z}\;\text{and}\;\pd(Z)<\infty\}.$
We point out that, for $\D=\p,$ we have that 
$\Psi_{[\D]}=\Psi,$ where $\Psi$ is the usual Igusa-Todorov function introduced in \cite{IT}.\\
For a subclass $\X\subseteq\modu\,(\Lambda)$ and a function $\Gamma:\modu\,(\Lambda)\to \mathbb{N},$ we introduce the $\Gamma$-dimension of 
$\X,$ defined by 
\begin{equation}\label{gamma}
\Gamma\mathrm{dim}(\X):=\sup\{\Gamma(X)\;:\;X\in\X\}.
\end{equation}
If $\X=\modu\,(\Lambda),$ we get the $\Gamma$-dimension $\Gamma\mathrm{dim}(\Lambda)$ of the algebra $\Lambda,$ which is $\Gamma\mathrm{dim}(\Lambda):=\Gamma\mathrm{dim}(\modu\,(\Lambda)).$ For the particular cases, when $\Gamma=\Phi_{[\D]}$ and $\Gamma=\Psi_{[\D]},$ we get the relative dimensions $\Phi_{[\D]}\text{dim}(\X)$ and 
$\Psi_{[\D]}\text{dim}(\X).$ Moreover, we have the homological dimensions  $\Phi_{[\D]}\text{dim}(\Lambda)$ and $\Psi_{[\D]}\text{dim}(\Lambda)$ of the given algebra $\Lambda.$\\
The main result of Section 3 can be stated as follows. For details, see Theorem \ref{thm_comparison} and Proposition \ref{previoLIT}.
\

{\bf Theorem A} Let $\Lambda$ be an Artin algebra and let $\D\subseteq\modu\,(\Lambda)$ be such that $\add\,\D=\D$ and 
$\Omega(\D)\subseteq \D.$  Then, for every $X\in \modu\,(\Lambda),$ the following statements hold true.
\begin{itemize}
\item[(a)] $\Phi(X)\leq \Phi_{[\D]}(X)+\Phidim(\D).$ 
\item[(b)] If $\Phidim(\D)=0$ then $\Psi(X)\leq \Psi_{[\D]}(X).$
\end{itemize}
\noindent
The additional condition $\Phidim(\D)=0,$ in the above theorem, is necessary in the sense that the inequality $\Psi(X)\leq\Psi_{[\D]}(X)+\Psidim(\D)$ does not hold in general, as can be seen in the example given in Remark \ref{contraejemplo}.
\

In Section 4, we do some applications of the developed theory. The main result of this section is the following theorem, see details in Theorem \ref{Phidimcota'}.
\

{\bf Theorem B} For an Artin algebra $\Lambda,$  a non-negative integer $n$ and  a left saturated subclass $\D$ in $\modu\,(\Lambda)$ such that $\Phidim(\D)=0,$ the following statements hold true. 
\begin{itemize}
\item[(a)] $\findim(\Lambda)\leq\Phidim(\Lambda)\leq\Psidim(\Lambda)\leq\resdim_\D(\Omega^n(\modu\,(\Lambda)))+n.$
\item[(b)] Let $\omega$ be closed  under direct summands and a $\D$-injective relative cogenerator in $\D.$  If $\Omega^n(\modu\,(\Lambda))\subseteq \D^\wedge,$ then 
$$\findim(\Lambda)\leq\Phidim(\Lambda)\leq\Psidim(\Lambda)\leq\id(\omega)+n.$$
\end{itemize}
\noindent
A similar result is also discussed, see Theorem \ref{Phidimcota}, if it is not assumed that $\Phidim(\D)=0.$ As an application of Theorem B, we can get the following result which is stated in Theorem \ref{AppCotilting}.
\

{\bf Theorem C} Let $\Lambda$ be an Artin algebra and $T\in\modu\,(\Lambda)$ be a cotilting $\Lambda$-module. Then, the following statements hold true. 
\begin{itemize}
\item[(a)] $\id(T)=\resdim_{{}^\perp T}(\modu\,(\Lambda)).$
\item[(b)] $\findim(\Lambda)\leq\Phidim(\Lambda)\leq\id(T)+\Phidim({}^\perp T).$
\item[(c)] If $\Phidim({}^\perp T)=0,$ then  $\findim(\Lambda)\leq\Phidim(\Lambda)\leq\Psidim(\Lambda)\leq \id(T).$
\end{itemize}

Another application of Theorem B is related with Gorenstein homological algebra. Indeed, the global Gorenstein projective dimension of an Artin algebra
 $\Lambda$ is $\mathrm{gl.Gpdim}(\Lambda):=\sup\{\Gpd(M)\;:\;M\in\modu\,(\Lambda)\},$ where $\Gpd(M)$ is the Gorenstein projective dimension of 
 $M.$ The following result generalises \cite[Theorem 4.7]{LM}, for a complete version see Theorem \ref{CorGPdim}.

{\bf Theorem D}  For an Artin algebra $\Lambda,$ with  $\mathrm{gl.Gpdim}(\Lambda)<\infty,$ we have that  
$$\findim(\Lambda)=\Phidim(\Lambda)=\Psidim(\Lambda)=\mathrm{gl.Gdim}(\Lambda)\leq \id({}_\Lambda\Lambda).$$

Finally, in Section 5, we give the definition of Lat Igusa-Todorov algebra which is presented in what follows.
\

{\bf Definition} An $n$-Lat-Igusa-Todorov algebra ($n$-LIT-algebra, for short), where $n$ is a non-negative integer, is an Artin algebra $\Lambda$ satisfying the following two conditions:
\begin{itemize}
\item[(a)]  there is some class $\D\subseteq\modu\,(\Lambda)$ such that $\add\,\D=\D,$  $\Omega(\D)\subseteq\D$ and $\Phidim\,(\D)=0;$
\item[(b)] there is some $V\in\modu\,(\Lambda)$ satisfying that each $M\in \modu\,(\Lambda)$ admits an exact sequence 
$$0\longrightarrow X_1\longrightarrow X_0\longrightarrow \Omega^nM\longrightarrow 0,$$ such that $X_1=V_1\oplus D_1$, $X_0=V_0\oplus D_0$, with $V_1,V_0\in \add\,V$ and $D_1,D_0\in \D.$
\end{itemize}
In case we need to specify the class $\D$ and the $\Lambda$-module $V,$ in the above definition, we say that $\Lambda$ is a $(n,V,\D)$-LIT-algebra.
\

It is also shown that the class of Lat Igusa-Todorov algebras strictly contains the class of Igusa-Todorov algebras. The main result of this section is the following one which can be seen in Theorem \ref{MainResult}.
\

{\bf Theorem E} If $\Lambda$ is a $(n,V,\D)$-LIT-algebra, then $$\findim\,(\Lambda)\leq \Psi_{[\D]}(V)+n+1<\infty.$$

\section{Preliminaries}

In this section, we collect some necessary material for the developing of this article. Some general notation will be given which are related with the relative projective (injective) and resolution (coresolution) dimensions with respect to some class $\X$ in $\modu\,(\Lambda),$ for any Artin algebra 
$\Lambda.$
\

We denote by $\p$ or $\p_\Lambda$ the class of all the finitely generated projective $\Lambda$-modules. For some 
$\X\subseteq\modu\,(\Lambda)$ and a non-negative integer $i,$ the $i$-th right orthogonal class of $\X$ is 
$\X^{\perp_i}:=\{M\in \modu\,(\Lambda)\,:\;\Ext^i_\Lambda(-,M)|_\X=0\};$ the  right  orthogonal category of $\X$ is 
$\X^\perp:=\cap_{i\geq 1}\X^{\perp_i}.$ Analogously, we have the $i$-th left orthogonal class ${}^{\perp_i}\X$ and the left orthogonal 
class ${}^{\perp}\X$ of $\X.$
\

Throughout this work, we will use freely some notations and considerations given in \cite{AB, BMPS}. Let $\X\subseteq\modu\,(\Lambda),$ it is said that $\X$ is {\bf pre-resolving} if it is closed under extensions and kernels of epimorphisms between its objects. We say that $\X$ is 
{\bf left thick} if $\X$ is pre-resolving and closed under direct summands in $\modu\,(\Lambda).$ A pre-resolving class which contains the class 
$\p,$ of projective $\Lambda$-modules, is called {\bf resolving}. A {\bf left saturated} class is a resolving class which is closed under direct summands in 
$\modu\,(\Lambda).$ Dually, we have the notions of pre-coresolving, right thick, coresolving and right saturated. For example, the class ${}^{\perp}\X$ is left saturated and $\X^\perp$ is right saturated in $\modu\,(\Lambda).$
\medskip

{\sc Relative homological dimensions.}  Let $\Lambda$ be an Artin algebra, $\X\subseteq\modu\,(\Lambda)$ and $M\in\modu\,(\Lambda).$ The {\bf relative projective dimension} of $M,$ with 
 respect to $\X,$ is 
  $$\pd_{\X}\,(M):=\min\{n\in\mathbb{N}\,:\,\Ext^j_\Lambda(M,-)|_{\X}=0  \text{ for any } j>n\},$$
 where $\min\emptyset:=\infty.$   Dually, we 
  have that $\id_{\X}\,(M)$ is the  {\bf{relative injective dimension}} of
  $M,$ with respect to $\X.$ Furthermore, for any  $\Y\subseteq\modu\,(\Lambda),$ we set $$\pd_\X\,(\Y):=\mathrm{sup}\,\{\pd_\X (Y)\;:\; Y\in\Y\}\text{ and }\id_\X\,(\Y):=\mathrm{sup}\,\{\id_\X(Y)\;:\; Y\in\Y\}.$$ 
 It can be shown, see \cite{AB}, that $\pd_\X\,(\Y)=\id_\Y\,(\X).$ For $\X=\modu\,(\Lambda),$ we have that $\pd_\X(\Y)=\pd\,(\Y)$ and   
 $\id_\X(\Y)=\id\,(\Y).$
\medskip

{\sc Relative cogenerators and generators.} Let $\Lambda$ be an Artin algebra and $(\X,\omega)$ be a pair of classes of objects in $\modu\,(\Lambda).$ The class $\omega$ 
is {\bf $\X$-injective} if $\id_\X(\omega)=0.$ It is said that $\omega$ is a {\bf relative cogenerator} in $\X$ if $\omega\subseteq\X$ and, for 
any $X\in\X,$ there is an exact sequence $0\to X\to W\to X'\to 0,$ with $W\in\omega$ and $X'\in\X.$ 
Dually, we have the notions of {\bf $\X$-projective} and {\bf relative generator} in $\X.$
\medskip
 
 {\sc Resolution and coresolution dimension.}
Let $\Lambda$ be an Artin algebra, $M\in\modu\,(\Lambda)$ and $\X\subseteq\modu\,(\Lambda).$ The $\X$-{\bf coresolution dimension} $\coresdim_\X\,(M)$ of $M$ is the minimal 
non-negative integer $n$ such that there is an exact sequence $$0\to M\to X_0\to X_1\to\cdots\to X_n\to 0$$ with $X_i\in\X$ for 
$0\leq i\leq n.$ If such $n$ does not exist, we set $\coresdim_\X\,(M):=\infty.$ We denote by $\mathcal{X}^{\vee}_n$ the  class of modules $M\in\modu\,(\Lambda)$ with $\coresdim_\X\,(M)\leq n$.  We also consider the class $\mathcal{X}^{\vee}:=\bigcup_{n\geq 0}\mathcal{X}^{\vee}_n.$ 
Dually, we have the $\X$-{\bf{resolution dimension}} $\resdim_\X\,(M)$ of $M,$ and the classes $\mathcal{X}^{\wedge}_n$  and 
$\mathcal{X}^{\wedge}:=\bigcup_{n\geq 0}\mathcal{X}^{\wedge}_n.$ 
Given a class $\Y\subseteq\modu\,(\Lambda),$ we set  
$$\coresdim_\X\,(\Y):=\mathrm{sup}\,\{\coresdim_\X\,(Y)\;
:\; Y\in\Y\},$$ and  $\resdim_\X\,(\Y)$ is defined dually.
\medskip

{\sc Cotorsion pairs.} The notion of cotorsion pair was introduced by L. Salce in \cite{S}. It is the analog of a torsion pair, where the functor 
$\Hom_\Lambda(-,-)$ is replaced by $\Ext^1_\Lambda(-,-).$
\

Let $\Lambda$ be an Artin algebra. A {\bf cotorsion pair} in $\modu\,(\Lambda)$ is a pair $(\X,\Y)$ of classes of objects in $\modu\,(\Lambda)$ satisfying that  
$\X={}^{\perp_1}\Y$ and  $\X^{\perp_1}=\Y.$ A cotorsion pair  $(\X,\Y)$ in $\modu\,(\Lambda)$ is {\bf complete} if, for any $M\in\modu\,(\Lambda),$ there are 
exact sequences $0\to Y'\to X'\to M\to 0$ and $0\to M\to Y\to X\to 0,$ where $X,X'\in\X$ and $Y,Y'\in\Y.$ Finally, a cotorsion pair  $(\X,\Y)$ in $\modu\,(\Lambda)$
is {\bf hereditary} if $\X$ is resolving and $\Y$ is coresolving.
\medskip

{\sc Cotilting modules.} Following Y. Miyashita in \cite{M}, we recall the definition of cotilting object in $\modu\,(\Lambda),$ for an Artin algebra $\Lambda.$ 
\

Let $T\in\modu\,(\Lambda).$ It is said that $T$ is {\bf cotilting} provided that: $\id\,T<\infty,$  $\Ext^i_\Lambda(T,T)=0,$ for any $i>0,$ and  
$\resdim_{\add\,T}(I)<\infty,$ for any injective $I\in\modu\,(\Lambda).$

\section{Generalised Igusa-Todorov Functions}

This section is devoted to an application of the Fitting's Lemma, a well known result proved by the German mathematician Hans Fitting in the 1930's. The ideas we are presenting here have their origin in the article \cite{IT}, where the functions that nowadays are known as Igusa-Todorov functions were firstly introduced. Let us start by recalling the Fitting's Lemma.

\begin{lem}\label{Fitting} Let $R$ be a noetherian ring. Consider a left $R$-module  $M$ and $f\in \End_R(M)$. Then, for any finitely generated $R$-submodule $X$ of $M$, there is a non-negative integer 
$$\eta_f(X) :=\min\{ k\in \mathbb{N}\,:\, f|_{f^{m}(X)}: f^{m}(X)\stackrel{\sim}{\rightarrow} f^{m+1}(X),\ \forall m\geq k\}.$$ 
Furthermore, for any $R$-submodule $Y$ of $X$, we have that $\eta_f(Y)\leq \eta_f(X)$.
\end{lem}

In \cite{IT},  K. Igusa and G. Todorov considered the free abelian group generated by the iso-classes 
$[X]:=\{Y\in\modu\,(\Lambda)\;:\;Y\simeq X\}$  of indecomposable non-projective $X\in\modu\,(\Lambda), $ where $\Lambda$ is an Artin algebra. They defined a group endomorphism $L[X]=[\Omega X],$ where $\Omega X$ is the first syzygy of the $\Lambda$-module $X.$ Moreover, for every 
$X\in\modu\,(\Lambda),$ it is associated the finitely generated subgroup $\langle X\rangle$ generated by the iso-classes of all the indecomposable non-projective elements in 
$\add\,X$. The image of $\langle X\rangle$ by $L$ is a free abelian group of finite rank. By means of Fitting's Lemma, it is possible to find a minimum integer $n$ such that the map $L: L^m(\langle X\rangle)\rightarrow L^{m+1}(\langle X\rangle)$ is an isomorphism, for all $m\geq n$. In this fashion, a function $\Phi: \modu\,(\Lambda)\rightarrow \mathbb{N}$ can be defined as $\Phi(X):=\eta_{L}(\langle X\rangle)$, where the right side corresponds to the integer obtained by using Fitting's Lemma. Associated to this function, there is another one denoted by $\Psi$ and these are commonly referred to as the Igusa-Todorov functions on $\modu\,(\Lambda).$ For more details on the definition of these functions and their properties, we refer the reader to \cite{IT}. Nevertheless we now take a few lines to recall a central result that we will use later on.

\begin{teo}\cite[Theorem 4]{IT}   \label{thm_psi}
Let $\Lambda$ be an Artin algebra and $0\rightarrow X\rightarrow Y\rightarrow Z\rightarrow 0$ be an exact sequence in $\modu\,(\Lambda)$ such that 
$\pd\,Z<\infty.$ Then $\pd\,Z\leq \Psi(X\oplus Y)+1.$ 
\end{teo}

Now, we start with the promised generalization of the Igusa-Todorov functions. In order to do that, we fix an Artin algebra $\Lambda.$ Let us look at the construction of the $\Phi$ function on $\modu\,(\Lambda)$ from a different and more general point of view. Denote by $\p=\p_\Lambda$ the class of all projective modules in $\modu\,(\Lambda).$  This class of modules is very particular since it is closed under direct sums and summands, that is $\p=\add\,\p=\add\, \Lambda.$

Denote by $K$ the free abelian group generated by the set of all the iso-classes $[M],$ for $M\in \modu\,(\Lambda),$ modulo the relation 
$[M]-[X]-[Y]$ if $M\simeq X\oplus Y.$ Note that $K$ is the free abelian group generated by the set $\ind(\Lambda)$ of all the iso-classes $[X]$ for all the indecomposable $X\in\modu\,(\Lambda).$ Observe that this abelian group $K$ is not the same as the abelian group $K$ considered by Igusa and Todorov in \cite{IT}, since the class $\p$ of all the projective modules in $\modu\,(\Lambda)$ vanishes in the original definition. From now on, for a subclass $\X\subseteq\modu\,(\Lambda)$ such that $\X=\add\,\X,$ we denote by $\langle\X\rangle$ the free abelian subgroup of $K$ generated by the set of all the iso-classes of indecomposable objects in $\X\cup \p.$ In the case of a single object $M\in\modu\,(\Lambda),$ we set $\langle M\rangle:=\langle\add(M)\rangle.$

Now, we fix a subclass $\D\subseteq\modu\,(\Lambda)$ such that $\D=\add\,\D$ and $\Omega\,\D\subseteq\D.$  Note that the first condition guarantees that $\langle\D\rangle$ is a direct summand of $K,$ and hence, the quotient group $K_\D:=K/\langle\D\rangle$ is  free abelian. In the case that $\D=\p,$  $K_\D$ coincides with the abelian group $K_\p$ considered by Igusa-Todorov in \cite{IT}. In this way, it is clear that both Igusa-Todorov functions $\Phi$ and $\Psi$ depend on the group generated by the class $\p$ of projectives. This raises the question of exactly how can be changed this class $\p$ of modules in order to define new Igusa-Todorov functions and what is the relation between them. In order to see a different solution of this question, by using relative homological algebra, we recommend the reader to see \cite{LM}. In what follows, we will address this issues.

The following result is a consequence of Fitting's Lemma, see  Lemma \ref{Fitting}.

\begin{lem}\label{LemaAux} Let $G$ be a  free abelian group,  $D$ be a subgroup of $G,$ $L\in\End_\Z(G)$ be such that $L(D)\subseteq D$ and let $k$  be a positive integer for which $L:L^k(D)\rightarrow D$ is a monomorphism. Then, for each finitely generated subgroup $X\subseteq G,$ we have  that
$$\eta_L(X)\leq \eta_{\overline{L}}(\overline{X})+k,$$
where $\overline{L}:G/D\rightarrow G/D,$ $g+D\mapsto L(g)+D,$ and $\overline{X}:=(X+D)/D.$
\end{lem}
\begin{proof} Let $X$ be a finitely generated subgroup of $G.$ According to Fitting's Lemma, there exists the smallest positive integer 
$\eta_L(X)$ such that $L:L^{m}(X)\rightarrow L^{m+1}(X)$ is an isomorphism for all $m\geq \eta_L(X)$. The same thing can be done for the endomorphism $\overline{L}$ and the subgroup $\overline{X}.$ 
\

Let $m\geq \eta_{\overline{L}}(\overline{X})+k.$ We assert that $L:L^{m}(X)\rightarrow L^{m+1}(X)$ is an isomorphism. Indeed, suppose that $L:L^{m}(X)\rightarrow L^{m+1}(X)$ is not an isomorphism. Then, there must be a non zero element $y\in L^{m}(X)$ such that $L(y)=0.$ This element can be written as $y=L^m(x)=L^k(L^{m-k}(x))$, for some $x\in X$. Since $L:L^k(D)\rightarrow D$ is a monomorphism, it follows that $L^{m-k}(x)\not\in D.$ Thus $L^{m-k}(x)+D=\overline{L}^{m-k}(x+D)$ is not zero and $\overline{L}^{k+1}(\overline{L}^{m-k}(x+D))$ has to be also different from zero, since the map $\overline{L}:\overline{L}^d(\overline X)\rightarrow \overline{L}^{d+1}(\overline X)$ is an isomorphism, for any $d\geq n_{\overline{L}}(\overline{X}).$ Moreover, this implies that $\overline{L}^{k+1}:\overline{L}^d(\overline X)\rightarrow \overline{L}^{d+k+1}(\overline X)$ is also an isomorphism. By setting $d:=m-k,$ we get a contradiction since $\overline{L}^{k+1}(\overline{L}^{m-k}(x+D))=\overline{L}(y+D)=0.$
\end{proof}

By going back from the group context to that of the category of modules $\modu\,(\Lambda),$ for an Artin algebra $\Lambda,$ we give the following central concept.

\begin{defi}\label{defi_phitecho}
Let $\Lambda$ be an Artin algebra and $\D\subseteq\modu\,(\Lambda)$ be a class of $\Lambda$-modules satisfying that $\add\,\D=\D$ and 
$\Omega(\D)\subseteq \D.$ Let $L:K\rightarrow K$ be the endomorphism defined by $L([X])=[\Omega\,X],$ and 
$\overline{L}:K_\D\to K_\D,$ $[X]+\langle\D\rangle \mapsto L([X])+\langle\D\rangle.$ For any $X\in\modu\,(\Lambda),$ we set 
$$\Phi_{[\D]}(X):=\eta_{\overline L}(\overline{\langle X\rangle}),$$
where $\overline{\langle X\rangle}:=(\langle X\rangle+\langle\D\rangle)/\langle\D\rangle.$ Note that $\overline{\langle X\rangle}$
is  the free abelian group generated by the set of all the iso-classes of indecomposable non-projective $\Lambda$-modules of $\add\,X$ that do not lie in 
$\D.$
\end{defi}

As a remark, let us just point out that this definition is a generalisation of the Igusa-Todorov function $\Phi$, which can be obtained by taking 
$\mathcal{D}=\p.$  The next theorem is a consequence of the above definition and Lemma \ref{LemaAux}.

\begin{teo}\label{thm_comparison}
Let $\Lambda$ be an Artin algebra and $\D\subseteq\modu\,(\Lambda)$ be a class of $\Lambda$-modules satisfying that $\add\,\D=\D$ and 
$\Omega(\D)\subseteq \D.$  Then, $\Phi(X)\leq \Phi_{[\D]}(X)+\Phidim(\D),$ for every $X\in \modu\,(\Lambda).$
\end{teo}

\begin{proof} We can assume that $k:=\Phidim(\D)<\infty.$ In order to apply Lemma \ref{LemaAux},  there is an additional condition involving the endomorphism $L:K\rightarrow K$ and the group $D:=\langle \D\rangle$, namely $L:L^m(D)\rightarrow D$ is a monomorphism for some non negative integer $m$. As one can guess, we will take $m$ to be the integer $k;$  and the reason it works is that, if $L:L^k(D)\rightarrow D$ is not monic, then there has to be a module $X\in \D$ such that $\Phi(X)\geq k+1$, contradicting the fact that $\Phidim(\D)=k$. Now we have all the hypothesis needed to apply Lemma \ref{LemaAux}, and thus the result follows. 
\end{proof}

At this moment, we would like to give, through Example \ref{ex_D} and Lemma \ref{EjemploBasico}, some substantial examples of classes of 
$\Lambda$-modules that could be taken as $\D$. It is worth mentioning that, in Definition \ref{defi_phitecho}, it is not required that the $\Phi$-dimension of $\D$ is finite, but the most important results in this work depend on this condition.

\begin{ex}\label{ex_D} Let $\Lambda$ be an Artin algebra.
\begin{itemize}
\item[(1)] The extremal cases are obtained by taking the classes $\p$ and $\modu\,(\Lambda)$. In the first case $\Phi_{[\p]}=\Phi$ and in the second one $\Phi_{[\modu\,(\Lambda)]}=0.$

\item[(2)] Let $T=T_\Lambda$ be a tilting module in $\modu\,(\Lambda^{op})$ such that $\pd(T)\leq 1$  and let $\Gamma:=\End(T_\Lambda).$ We have in
$\modu\,(\Gamma^{op})$ the torsion free class $\mathcal{Y}(T):=\{Y\in\modu\,(\Gamma^{op})\;:\; \mathrm{Tor}^{\Gamma^{op}}_1(Y,T)=0\},$ which contains the projectives and it is closed under submodules. This implies that  $\D:=\mathcal{Y}(T)$ is closed under syzygies, direct sums and direct summands. We only need to see when the $\Phi$-dimension of $\mathcal{Y}(T)$ is finite, but the most interesting thing happens here: it follows, from a more general result proved in \cite{FLM}, that $\Phidim(\Lambda)<\infty\; \Leftrightarrow\; \Phidim(\Gamma)<\infty$. Now, it is easy to conclude that the last one is equivalent to $\Phidim(\mathcal{Y}(T))<\infty.$

\item[(3)] Let $\D:=\p^{<\infty}$ be the subclass of $\modu\,(\Lambda)$ consisting of all the $\Lambda$-modules with finite projective dimension. It is clear that $\D$ satisfies the conditions imposed in Definition \ref{defi_phitecho}. In this case, the finiteness of $\Phidim(\D)$ is equivalent to the finiteness of the finitistic dimension of $\Lambda.$ We can also take $\D:=\p^{\leq n}$ to be the subclass of $\modu\,(\Lambda)$ consisting of all the $\Lambda$-modules with projective dimension at most $n;$ and here we get $\Phidim(\D)\leq n.$   
\end{itemize}
\end{ex}

The following two classes are our prototype for the class $\D\subseteq\modu\,(\Lambda)$ given in Definition \ref{defi_phitecho}. Moreover, they satisfy the additional condition that $\Phidim(\D)=0$. In order to do that, we recall the notion of Gorenstein projective object in $\modu\,(\Lambda),$ for some Artin algebra $\Lambda.$ 
\

An object $M\in\modu\,(\Lambda)$ is {\bf Gorenstein projective} if there is an acyclic complex 
$\mathbf{P}_\bullet=(P_m, d^{\mathbf{P}}_m:P_m\to P_{m-1})_{m\in\mathbb{Z}}$ of objects in $\p_\Lambda$
such that $M=\Ker\,d^{\mathbf{P}}_0$ and $\Hom_\Lambda(\mathbf{P}_\bullet,Q)$ is an acyclic complex for any $Q\in\p_\Lambda.$ The class of all the Gorenstein projective objects in $\modu\,(\Lambda)$ is denoted by $\Gproj(\Lambda).$

\begin{lem}\label{EjemploBasico} Let $\Lambda$ be an Artin algebra.  Then,  for $\D:=\Gproj(\Lambda)$ or $\D:={}^\perp\Lambda,$ it follows that  $\Phidim\,(\D)=0$ and $\D$ is left saturated. In particular, $\add\,\D=\D$ and $\Omega(\D)\subseteq \D.$
\end{lem}
\begin{proof}  By \cite[Theorem 2.5]{Holm} and \cite[Remark 7.4 (a)]{LM}, it follows that the class $\Gproj(\Lambda)$ is left saturated. Moreover, \cite[Proposition 2.3]{Holm} implies that $\Gproj(\Lambda)\subseteq {}^\perp\Lambda.$ On the other hand, by using \cite[Corollary 4.2]{LM}, we get that $\Phidim\,({}^{\perp}\Lambda)=0;$  and thus, $\Phidim\,(\Gproj(\Lambda))=0.$ Finally, it is  well known that the class ${}^\perp\Lambda$  is left saturated.
\end{proof}

At this point, we would like to formulate in this context some fundamental definitions and properties that arise naturally from the developed theory. First thing that one can think of is the generalisation of the second Igusa-Todorov function $\Psi,$ and then to check out if all the known properties hold true with the obvious modifications.

For any subclass $\X\subseteq\modu\,(\Lambda),$ the finitistic dimension of $\X$ is $\findim(\X):=\sup\,\{\pd\,X\;:\; X\in \X\;\text{and}\;\pd\,X<\infty\}.$ For any $M,N\in\modu\,(\Lambda),$ we write $M|N$ if $M$ is a direct summand of $N.$

\begin{defi}\label{defi_psitecho}
Let $\Lambda$ be an Artin algebra and $\D\subseteq\modu\,(\Lambda)$ be a class of $\Lambda$-modules satisfying that $\add\,\D=\D$ and 
$\Omega(\D)\subseteq \D.$ For any $X\in\modu\,(\Lambda),$ we set
$$\Psi_{[\D]}(X):=\Phi_{[\D]}(X) + \findim\,(\{Z\in \modu\,(\Lambda)\;:\; Z|
\Omega^{\Phi_{[\D]}(X)}(X)\}).$$
\end{defi}
Note that $$\findim\,(\{Z\in \modu\,(\Lambda)\;:\; Z|\Omega^{\Phi_{[\D]}(X)}(X)\})=\findim\,(\add\,\Omega^{\Phi_{[\D]}(X)}(X)).$$

The next propositions are a compendium of fundamental properties satisfied by these generalised Igusa-Todorov functions. We recall, from (\ref{gamma}) in the introduction, the definitions of $\Phi_{[\D]}\mathrm{dim}(\X)$ and $\Psi_{[\D]}\mathrm{dim}(\X)$, for $\X\subseteq \modu\,(\Lambda).$

\begin{pro}\label{prop_propiedades}
Let $\Lambda$ be an Artin algebra and $\D\subseteq\modu\,(\Lambda)$ be a class of $\Lambda$-modules satisfying that $\add\,\D=\D$ and 
$\Omega(\D)\subseteq \D.$ Then,  the following statements hold true, for $X,Y,M\in \modu\,(\Lambda).$ 
\begin{itemize}
\item[(a)] If $M\in \D\cup \p$, then $\Phi_{[\D]}(M)=0$ and $\Phi_{[\D]}(X\oplus M)=\Phi_{[\D]}(X).$
\item[(b)] $\Phi_{[\D]}(X)\leq \Phi_{[\D]}(X\oplus Y).$
\item[(c)] $\Phi_{[\D]}\mathrm{dim}(\add\,X)=\Phi_{[\D]}(X).$
\item[(d)] If $\pd\,X<\infty$ and $\Phidim(\D)=0$, then $\Phi_{[\D]}(X)=\Phi(X)=\pd\,X.$
\item[(e)] Let $Z|\Omega^t(X)$ be such that $0\leq t\leq \Phi_{[\D]}(X)$ and $\pd\,Z<\infty.$ Then\\ $\pd\,Z+t\leq \Psi_{[\D]}(X).$
\item[(f)] $\Psi_{[\D]}(X)\leq \Psi_{[\D]}(X\oplus Y).$
\item[(g)] $\Psi_{[\D]}\mathrm{dim}(\add\,X)=\Psi_{[\D]}(X).$ 
\end{itemize}
\end{pro}

\begin{proof}  (a) It follows immediately from Definition \ref{defi_phitecho}.
\

 (b)  We must notice firstly that $\overline{\langle X\rangle}$ is a subgroup of $\overline{\langle X\oplus  Y\rangle}.$ Hence we get the inequality from Fitting's Lemma. 
 \
 
 (c) Let $Z\in\add\,X.$ Then, there exists   $Z'\in\modu\,(\Lambda)$  such that $Z\oplus Z'=X^s,$ for some non-negative integer $s.$ From Definition \ref{defi_phitecho}, we get that $\Phi_{[\D]}(X^s)=\Phi_{[\D]}(X);$ and if we apply (b),  we  obtain that 
  $\Phi_{[\D]}(Z)\leq \Phi_{[\D]}(X).$
 \ 
  
(d) Note firstly that, since $\Phidim(\D)=0,$ the class $\D$ does not contain $\Lambda$-modules of 
 finite projective dimension that are not projective. Now, suppose $\pd\,X=n.$ Then 
 $\langle \Omega^n X\rangle=\overline{\langle \Omega^n X\rangle}=0,$ while $\langle \Omega^{n-1} X\rangle$ and $\overline{\langle \Omega^{n-1} X\rangle}$ are both non zero, proving that $\Phi_{[\D]}(X)=\Phi(X)=n.$ 
 \
 
 (e) We point out firstly, that $\Omega^{\Phi_{[\D]}(X)-t}(Z)|\Omega^{\Phi_{[\D]}(X)}(X);$ and that follows, since the syzygy operator commutes with direct sums and $\pd\,\Omega^{\Phi_{[\D]}(X)-t}(Z)$ is finite due to the finiteness of $\pd\,Z.$  Moreover, we have that 
 $$\pd\,Z\leq \pd\,\Omega^{\Phi_{[\D]}(X)-t}(Z) + \Phi_{[\D]}(X)-t.$$ 
 On the other hand, note that $\pd\,\Omega^{\Phi_{[\D]}(X)-t}(Z)\leq \findim\,(\add\,\Omega^{\Phi_{[\D]}(X)}(X)).$ Then, by combining the above inequalities, we get $$\pd\,Z+t\leq \findim\, (\add\,\Omega^{\Phi_{[\D]}(X)}(X)) + \Phi_{[\D]}(X) = \Psi_{[\D]}(X).$$
 
 (f) Let $Z|\Omega^{\Phi_{[\D]}(X)}(X)$ with $\pd\, Z$  finite. From  (b) we know that $\Phi_{[\D]}(X)\leq \Phi_{[\D]}(X\oplus Y)$ and thus by (e), since $Z$ is also a direct summand of $\Omega^{\Phi_{[\D]}(X)}(X\oplus Y)$, we get  that
 $$\pd\,Z+\Phi_{[\D]}(X)\leq \Psi_{[\D]}(X\oplus Y).$$ 
 Finally, by taking the supremum  in the left side of the preceding inequality, we get the desired inequality in (f).
 \
 
 (g) Let $Z\in\add\, X.$ Then, there exists   $Z'\in\modu\,(\Lambda)$  such that $Z\oplus Z'=X^s,$ for some non-negative integer $s.$ Note that $\Phi_{[\D]}(X^s)=\Phi_{[\D]}(X)$ as can be seen in the proof of (c). Then,  we get the equalities
\begin{equation*}
\begin{split}
\Psi_{[\D]}(X^s) & =\Phi_{[\D]}(X^s)+\findim\,(\add\,\Omega^{\Phi_{[\D]}(X^s)}(X^s))\\
                       &=\Phi_{[\D]}(X)+\findim\,(\add\,\Omega^{\Phi_{[\D]}(X)}(X^s))\\
                       &=\Phi_{[\D]}(X)+\findim\,(\add\,\Omega^{\Phi_{[\D]}(X)}(X))\\
                       &=\Psi_{[\D]}(X).
\end{split}
\end{equation*} 
 Therefore, by (f), it follows that $\Psi_{[\D]}(Z)\leq \Psi_{[\D]}(X^s)=\Psi_{[\D]}(X),$ proving (g).
\end{proof}

For the next proposition we require the additional condition $\Phidim(\D)=0.$

\begin{pro}\label{previoLIT}  Let $\Lambda$ be an Artin algebra and $\D\subseteq\modu\,(\Lambda)$ be such that $\add\,\D=\D,$  
$\Omega(\D)\subseteq\D$ and $\Phidim(\D)=0.$ Then, the following statements hold true.
\begin{itemize}
\item[(a)] $\Psi(X)\leq\Psi_{[\D]}(X),$ for any $X\in\modu\,(\Lambda).$
\item[(b)] $\Psi_{[\D]}(X\oplus D)=\Psi_{[\D]}(X),$ for any $X\in\modu\,(\Lambda)$ and $D\in\D.$
\item[(c)] $\Psi_{[\D]}\mathrm{dim}(\D)=0.$
\end{itemize}
\end{pro}
\begin{proof} (a) Let $X\in\modu\,(\Lambda).$ Consider some $Z|\Omega^{\Phi(X)}(X)$ with $\pd\,Z<\infty.$ Since $\Phidim(\D)=0,$ we get from Theorem \ref{thm_comparison} that $0\leq \Phi(X)\leq\Phi_{[\D]}(X).$ Thus, by Proposition \ref{prop_propiedades} (e), it 
follows that $\pd\,Z+\Phi(X)\leq\Psi_{[\D]}(X).$ Hence $\Psi(X)\leq\Psi_{[\D]}(X).$
\

(b) Let $X\in\modu\,(\Lambda)$ and $D\in\D.$ Using Proposition \ref{prop_propiedades} (d), the fact that $\Phidim(\D)=0$ and 
$\Omega(\D)\subseteq\D,$ we get that 
$$\findim\,(\add\,\Omega^{\Phi_{[\D]}(X)}(X\oplus D))=\findim\,(\add\,\Omega^{\Phi_{[\D]}(X)}(X)).$$
On the other hand, Proposition \ref{prop_propiedades} (a) gives us $\Phi_{[\D]}(X\oplus D)=\Phi_{[\D]}(X).$ Then
\begin{equation*}
\begin{split}
\Psi_{[\D]}(X\oplus D) & =\Phi_{[\D]}(X\oplus D)+\findim\,(\add\,\Omega^{\Phi_{[\D]}(X\oplus D)}(X\oplus D))\\
                       &=\Phi_{[\D]}(X)+\findim\,(\add\,\Omega^{\Phi_{[\D]}(X)}(X\oplus D))\\
                       &=\Phi_{[\D]}(X)+\findim\,(\add\,\Omega^{\Phi_{[\D]}(X)}(X))\\
                       &=\Psi_{[\D]}(X),
\end{split}
\end{equation*} 
proving (b).
\

(c) Let $D\in\D.$ From Proposition \ref{prop_propiedades} (a), we get that $\Phi_{[\D]}(D)=0.$ On the other hand, by Proposition \ref{prop_propiedades} (d), we conclude that $\findim\,(\add\,D)=0.$ Finally, we have
$\Psi_{[\D]}( D)=\Phi_{[\D]}(D)+\findim\,(\add\,\Omega^{\Phi_{[\D]}(D)}(D))=\findim\,(\add\,D)=0.$      
\end{proof}

\begin{remark} \label{contraejemplo}
\begin{itemize}
\item[(a)] We would like to point out that, in contrast with Theorem \ref{thm_comparison}, the additional condition $\Phidim(\D)=0$ is necessary in the sense that the inequality $\Psi(X)\leq\Psi_{[\D]}(X)+\Psidim(\D)$ does not hold in general, as can be shown in the next example. Consider the quotient path $k$-algebra $\Lambda:=kQ/J^2$ given by the following quiver $Q$ 
$$\begin{tikzcd}
    Q: & 6 & 5\ar{l} & 4\ar{l} & 1\ar[loop, out=120, in=57, distance=2em]{}\ar{r}\ar{l} & 2\ar{r} & 3,
\end{tikzcd}$$
where $J$ is the ideal of $kQ$ generated by all the arrows of $Q.$

Let $\D := \add(\Lambda\oplus S_2)$ and  $X:=S_1\oplus S_2\in\modu\,(\Lambda).$ By direct computation we obtain $\Psi(X)=3$, while $\Psi_{[\D]}(X)=1$ and $\Psidim(\D)=1.$ Thus, the inequality $\Psi(X)\leq\Psi_{[\D]}(X)+\Psidim(\D)$ does not hold. The reason lies on the difficulty in controlling the projective dimension of the direct summands of syzygies of $\Lambda$-modules with infinite projective dimension.

\item[(b)] One more thing we can say about the relationship between $\Psi$ and $\Psi_{[\D]}$, apart from Proposition \ref{previoLIT}, is the following: for $X\in \mod\,(\Lambda),$ we have that $\Phi(X)\leq \Phi_{[\D]}(X)$ if and only if $\Psi(X)\leq \Psi_{[\D]}(X).$ The proof of this claim follows directly from the definition, since the fact that $\Phi(X)\leq \Phi_{[\D]}(X)$ implies that $$\findim\,(\add\,\Omega^{\Phi(X)}(X))-\findim\,(\add\,\Omega^{\Phi_{[\D]}(X)}(X))\leq \Phi_{[\D]}(X)-\Phi(X).$$
\end{itemize}
\end{remark}

The next two properties are motivated by their analogues in \cite{HLM1}.

\begin{pro}\label{previoLIT2}
Let $\Lambda$ be an Artin algebra and $\D\subseteq\modu\,(\Lambda)$ be such that $\add\,\D=\D$ and $\Omega(\D)\subseteq\D.$ Then,  for any $X\in \modu\,(\Lambda),$ we have that:
\begin{itemize}
\item[(a)] $\Phi_{[\D]}(X)\leq \Phi_{[\D]}(\Omega X)+1;$
\item[(b)] $\Psi_{[\D]}(X)\leq \Psi_{[\D]}(\Omega X)+1.$
\end{itemize}
\end{pro}
 
 \begin{proof}
  (a) The case $\Phi_{[\D]}(X)=0$ is trivial. Thus, we can assume that $\Phi_{[\D]}(X)>0$. Note that $\overline{L}(\overline{\langle X\rangle})$ is a subgroup of $\overline{\langle \Omega X\rangle},$ and hence  by Fitting's Lemma we get that $\Phi_{[\D]}(X)-1=\eta_{\overline L}(\overline{L}(\overline{\langle X\rangle}))\leq \eta_{\overline L}(\overline{\langle \Omega X\rangle})=\Phi_{[\D]}(\Omega X).$
  \
  
(b) Let $Z|\Omega^{\Phi_{[\D]}(X)}(X)$ with $\pd\,Z<\infty.$ Consider a non negative integer $t$ such that $\Phi_{[\D]}(X)+t= \Phi_{[\D]}(\Omega X)+1.$ Therefore $\Omega^t Z|\Omega^{\Phi_{[\D]}(\Omega X)}(\Omega X)$ and $\pd\,\Omega^t Z<\infty.$ By combining all of the above, and the fact that $\pd\,Z\leq \pd\,\Omega^tZ+t,$  we get
$$\pd\,Z\leq \findim\,(\add\,\Omega^{\Phi_{[\D]}(\Omega X)}(\Omega X))+\Phi_{[\D]}(\Omega X)+1-\Phi_{[\D]}(X).$$ Furthermore, by definition, the second term in the above inequality is equal to  
\begin{center}
$\Psi_{[\D]}(\Omega X)+1-\Phi_{[\D]}(X).$ 
\end{center}
Hence, $\Psi_{[\D]}(X)=\findim\,(\add\, \Omega^{\Phi_{[\D]} (X)}( X))+\Phi_{[\D]}(X)\leq \Psi_{[\D]}(\Omega X)+1.$
 \end{proof}

\section{Some applications of the Igusa-Todorov functions}

In this section, we do some applications of the developed theory in order to relate the different homological dimensions which have been discussed through the paper.

 \begin{lem}\label{previoLIT3} For an Artin algebra $\Lambda,$ a non-negative integer $n$ and  a resolving class $\D\subseteq\modu\,(\Lambda),$ the following statements hold true.
\begin{itemize}
\item[(a)] $\Omega^m(\D^\wedge_n)\subseteq \D$ for any $m\geq n.$
\item[(b)] $\resdim_{\D}(\Omega^n(\modu\,(\Lambda)))\leq k\; \Leftrightarrow\; \Omega^{n+k}(\modu\,(\Lambda))\subseteq \D.$
\end{itemize} 
 \end{lem} 
 \begin{proof} Note firstly that $\D$ is closed under isomorphisms and $\Omega^m(\D)\subseteq \D,$ for any $m\geq 0,$ since $\D$ is resolving.
 \
 
 (a) It is enough to show that $\Omega^n(\D^\wedge_n)\subseteq \D.$ We proceed by induction on $n.$  If $C\in \D^\wedge_0,$ then 
 $\Omega^0(C)=C\simeq D\in \D.$
 \
 
 Let $n\geq 1.$ Assume, by inductive hypothesis, that $\Omega^{n-1}(\D^\wedge_{n-1})\subseteq \D.$ Let $C\in \D^\wedge_n.$ Then, there is an exact sequence $0\to K\to D\to C\to 0,$ where $K\in \D^\wedge_{n-1}$ and $D\in\D.$ Consider the following exact and commutative diagram in $\modu\,(\Lambda)$
$$\xymatrix@-0.5pc{&& 0\ar[d] & 0\ar[d] \\
&& \Omega(C)\ar[d]\ar@{=}[r] & \Omega(C)\ar[d] \\
0\ar[r] & K\ar[r]\ar@{=}[d] & U\ar[r]\ar[d] & P\ar[r]\ar[d] & 0 \\ 
0\ar[r] & K\ar[r] & D\ar[r]\ar[d] & C\ar[r]\ar[d] & 0 \\
&& 0 & 0, }$$
where $P\to C$ is the projective cover of $C.$  In particular, the middle row splits and thus $U=K\oplus P.$ Moreover 
$\Omega^{n-1}(U)\simeq\Omega^{n-1}(K)\in\Omega^{n-1}(\D^\wedge_{n-1})\subseteq \D.$
\

Now, let us take the middle column of the above diagram and apply the Horseshoe lemma to obtain the exact sequence
$$0\rightarrow \Omega^n(C)\rightarrow \Omega^{n-1}(U)\oplus Q\rightarrow \Omega^{n-1}(D)\rightarrow 0,$$
with $Q$ a projective $\Lambda$-module. Since $\D$ is resolving,  we get $\Omega^n(C)\in \D.$
\

(b)  Let $\resdim_{\D}(\Omega^n(\modu\,(\Lambda)))\leq k.$ In particular  $\Omega^n(\modu\,(\Lambda))\subseteq \D^\wedge_k$ 
and thus, by (a), we get $\Omega^{n+k}(\modu\,(\Lambda))\subseteq \Omega^k(\D^\wedge_k)\subseteq\D.$
\

Assume that $ \Omega^{n+k}(\modu\,(\Lambda))\subseteq \D.$ Let $C\in \modu\,(\Lambda).$ By using the minimal projective resolution of $C,$ we  get the exact sequence 
$$0\rightarrow \Omega^{n+k}(C)\rightarrow P_{n+k-1}\to\cdots\rightarrow P_{n+1}\rightarrow P_n\rightarrow \Omega^{n}(C)\rightarrow 0,$$ where $P_i$ is projective for each $i.$ Since $\Omega^{n+k}(C)\in\D$ and $\p\subseteq\D,$ we get that $\resdim_\D(\Omega^n(C))\leq k.$
 \end{proof}
 
 \begin{lem}\label{previoLIT3(c)} Let $\Lambda$ be an Artin algebra,   $\D\subseteq\modu\,(\Lambda)$ be closed under extensions and direct summands, and let $\omega$ be closed under direct summands and a $\D$-injective relative cogenerator in $\D.$  If $\Omega^n(\modu\,(\Lambda))\subseteq \D^\wedge$ for a non-negative integer $n,$  then 
$$\id(\omega)-n\leq \resdim_\D(\Omega^n(\modu\,(\Lambda)))\leq \id(\omega).$$
 \end{lem} 
 \begin{proof} Let $\Omega^n(\modu\,(\Lambda))\subseteq \D^\wedge,$ $\alpha:=\resdim_\D(\Omega^n(\modu\,(\Lambda)))$ and $M\in\modu\,(\Lambda).$ From \cite[Proposition 2.1]{AB}, we get that 
$\pd_\omega(\Omega^n(M))=\resdim_\D(\Omega^n(M)).$  Therefore, 
by using that  $\Ext^j_\Lambda(\Omega^n(M),-)|_\omega\simeq\Ext^{j+n}_\Lambda(M,-)|_\omega,$ we get that 
$$\pd_\omega(\Omega^n(M))\leq\pd_\omega(M)\leq \pd_\omega(\Omega^n(M))+n\leq \alpha+n.$$
Thus, the result follows since $\pd_\omega(\Omega^n(\modu\,(\Lambda)))=\alpha$ and $\pd_\omega(\modu\,(\Lambda))=\id(\omega).$
\end{proof}
 
 As a first application of the developed theory, we get the following result. 
 
\begin{teo}\label{Phidimcota} For an Artin algebra $\Lambda,$  a non-negative integer $n$ and  a left saturated subclass $\D$ in $\modu\,(\Lambda),$ the following statements hold true. 
\begin{itemize}
\item[(a)] $\findim(\Lambda)\leq\Phidim(\Lambda)\leq\resdim_\D(\Omega^n(\modu\,(\Lambda)))+\Phidim(\D)+n.$
\item[(b)] Let $\omega$ be closed  under direct summands and a $\D$-injective relative cogenerator in $\D.$  If $\Omega^n(\modu\,(\Lambda))\subseteq \D^\wedge,$ then 
$$\findim(\Lambda)\leq\Phidim(\Lambda)\leq\id(\omega)+\Phidim(\D)+n.$$
\end{itemize}
\end{teo} 
 \begin{proof} (a) We can assume that $k:=\resdim_\D(\Omega^n(\modu\,(\Lambda)))$ and $d:=\Phidim(\D)$ are both finite.
 \
 
 By Lemma \ref{previoLIT3} (b) we know that $\Omega^{n+k}(\modu\,(\Lambda))\subseteq \D.$ Furthermore, by Proposition \ref{previoLIT2} (a) and 
 Proposition \ref{prop_propiedades} (a), it follows that $\Phi_{[\D]}\mathrm{dim}(\modu\,(\Lambda)))\leq n+k.$ Finally, from Theorem \ref{thm_comparison}
  we get $\Phidim(\Lambda)\leq \Phi_{[\D]}\mathrm{dim}(\modu\,(\Lambda)))+d;$ proving the result. 
  \
  
  (b)  It follows from (a) and Lemma \ref{previoLIT3(c)}.
 \end{proof}

\begin{teo}\label{Phidimcota'} For an Artin algebra $\Lambda,$  a non-negative integer $n$ and  a left saturated subclass $\D$ in $\modu\,(\Lambda)$ such that $\Phidim(\D)=0,$ the following statements hold true. 
\begin{itemize}
\item[(a)] $\findim(\Lambda)\leq\Phidim(\Lambda)\leq\Psidim(\Lambda)\leq\resdim_\D(\Omega^n(\modu\,(\Lambda)))+n.$
\item[(b)] Let $\omega$ be closed  under direct summands and a $\D$-injective relative cogenerator in $\D.$  If $\Omega^n(\modu\,(\Lambda))\subseteq \D^\wedge,$ then 
$$\findim(\Lambda)\leq\Phidim(\Lambda)\leq\Psidim(\Lambda)\leq\id(\omega)+n.$$
\end{itemize}
\end{teo}
\begin{proof} (a) Let $k:=\resdim_\D(\Omega^n(\modu\,(\Lambda)))$ be finite. Then, by Lemma \ref{previoLIT3} (b),  we get that $\Omega^{n+k}(\modu\,(\Lambda))\subseteq \D.$ Hence, by Proposition \ref{previoLIT} (a), Proposition \ref{previoLIT2} (b) and Proposition 
\ref{previoLIT} (c), the following inequalities hold true
\begin{align*}
\Psidim(\Lambda) &\leq \Psi_{[\D]}(\modu\,(\Lambda))\\
                           &\leq \Psi_{[\D]}(\Omega^{n+k}(\modu\,(\Lambda)))+n+k\\
                           &\leq \Psi_{[\D]}(\D)+n+k = n+k.
\end{align*}
  
  (b)  It follows from (a) and Lemma \ref{previoLIT3(c)}.
\end{proof}

We recall that, for a given Artin algebra $\Lambda$ and $M\in\modu\,(\Lambda),$ the resolution dimension $\Gpd(M):=\resdim_{\Gproj(\Lambda)}(M)$ is 
known as the Gorenstein projective dimension  of $M.$ Moreover, the  global Gorenstein projective dimension of $\Lambda$ is 
$\mathrm{gl.Gpdim}(\Lambda):=\sup\{\Gpd(M)\;:\;M\in\modu\,(\Lambda)\},$ and the finitistic Gorenstein projective dimension of $\Lambda$ is 
$\mathrm{fin.Gpdim}(\Lambda):=\sup\{\Gpd(M)\;:\;M\in\Gproj(\Lambda)^\wedge\}.$
\

The following result generalises \cite[Theorem 4.7]{LM}.

\begin{teo}\label{CorGPdim}  For an Artin algebra $\Lambda,$ the following statements hold true.
\begin{itemize}
\item[(a)] Let $n$ be any non-negative integer $n.$ Then
$$\mathrm{fin.Gpdim}(\Lambda)=\findim(\Lambda)\leq\Phidim(\Lambda)\leq\Psidim(\Lambda)\leq\Gpd(\Omega^n(\modu\,(\Lambda)))+n.$$
\item[(b)] Let $\mathrm{gl.Gpdim}(\Lambda)$ be finite. Then 
$$\findim(\Lambda)=\Phidim(\Lambda)=\Psidim(\Lambda)=\mathrm{gl.Gpdim}(\Lambda)\leq \id({}_\Lambda\Lambda).$$
\end{itemize}
\end{teo}
\begin{proof} (a)  Note firstly, that the pair $(\p_\Lambda, \p_\Lambda)$ satisfies the needed conditions in \cite[Theorem 4.23]{BMS}, for the abelian category $\modu\,(\Lambda);$ and thus $\mathrm{fin.Gpdim}(\Lambda)=\findim(\Lambda).$
\

Consider the class $\D:=\Gproj(\Lambda)\subseteq\modu\,(\Lambda).$ Then, the item (a) follows from Lemma 3.7 and Theorem \ref{Phidimcota'} (a).
\

(b) Since $\mathrm{gl.Gpdim}(\Lambda)$ is finite, we get that $\mathrm{fin.Gpdim}(\Lambda)=\mathrm{gl.Gpdim}(\Lambda).$ 
 Finally, the item (b) follows from (a), by taking $n=0,$ and Theorem \ref{Phidimcota'} (b) due to the fact that $\p_\Lambda$ is a $\Gproj(\Lambda)$-injective relative cogenerator in $\Gproj(\Lambda)$ and $\modu\,(\Lambda)=\Gproj(\Lambda)^\wedge.$
\end{proof}

By using hereditary cotorsion pairs in $\modu\,(\Lambda),$ there is another consequence of Theorems \ref{Phidimcota} and \ref{Phidimcota'}. In order to state this result, we need the following lemma.

\begin{lem}\label{cothered1}  Let $\Lambda$ be an Artin algebra and $(\X,\Y)$ be a hereditary complete cotorsion pair in 
$\modu\,(\Lambda).$ Then,  $\X$ is a left saturated class in $\modu\,(\Lambda)$ and $\omega:=\X\cap\Y$ is an $\X$-injective relative cogenerator in 
$\X.$
\end{lem}
\begin{proof} By  using that $\X={}^{\perp_1}\Y={}^\perp\Y,$ it can be shown that $\X$ is left saturated. Since $(\X,\Y)$ is a complete cotorsion pair, it follows that $\omega$ is a relative cogenerator in $\X.$ Finally, the fact that $\id_\X(\Y)=0$ implies that $\id_\X(\omega)=0.$
\end{proof}

\begin{cor}\label{CotPhidimcota} Let $\Lambda$ be an Artin algebra and $(\X,\Y)$ be a hereditary complete cotorsion pair in 
$\modu\,(\Lambda)$ such that $\Omega^n(\modu\,(\Lambda))\subseteq \X^\wedge,$ for some non-negative integer $n.$ Then, for $\omega:=\X\cap\Y,$  the following statements hold true.
\begin{itemize}
\item[(a)] $\findim(\Lambda)\leq\Phidim(\Lambda)\leq\id(\omega)+\Phidim(\X)+n.$
\item[(b)] If $\Phidim(\X)=0,$ then  $\findim(\Lambda)\leq\Phidim(\Lambda)\leq\Psidim(\Lambda)\leq\id(\omega)+n.$
\end{itemize}

\end{cor}
\begin{proof} By taking $\D:=\X$ and $\omega:=\X\cap\Y,$ the result follows from Theorem \ref{Phidimcota} (b), Theorem \ref{Phidimcota'} (b) and Lemma \ref{cothered1}.
\end{proof}

A special kind of hereditary cotorsion pairs in $\modu\,(\Lambda)$ are the ones associated to cotilting $\Lambda$-modules. In the following lemma we bring out to the light the needed properties of these cotilting cotorsion pairs in order to proof the theorem below. 

\begin{lem}\label{LemaTCH} Let $\Lambda$ be an Artin algebra and $T\in\modu\,(\Lambda)$ be a cotilting $\Lambda$-module. Then, the following statements hold true.
\begin{itemize}
\item[(a)] $({}^\perp T,(\add\,T)^\wedge)$ is a hereditary complete cotorsion pair in $\modu\,(\Lambda).$
\item[(b)]  $\add\,T={}^\perp T\cap(\add\,T)^\wedge$ and $({}^\perp T)^\wedge=\modu\,(\Lambda).$
\item[(c)] $\add(T)$ is ${}^\perp T$-injective and a relative cogenerator in the left saturated class ${}^\perp T.$
\end{itemize} 
\end{lem} 
\begin{proof}
Since $T$ is cotilting, by the proof of \cite[Theorem 5.17]{BMPS}, it can be seen that the cotorsion pair $({}^\perp T,(\add\,T)^\wedge)$ is hereditary and complete in $\modu\,(\Lambda),$  $\add\,T={}^\perp T\cap(\add\,T)^\wedge$ and $({}^\perp T)^\wedge=\modu\,(\Lambda).$ In particular, by Lemma \ref{cothered1}, we obtain that ${}^\perp T$ is left saturated and 
$\add(T)$ is a ${}^\perp T$-injective and a relative cogenerator in ${}^\perp T.$ 
\end{proof}

Now, we are ready to state and prove the latest application, in this section, of Theorem \ref{Phidimcota}.

\begin{teo}\label{AppCotilting} Let $\Lambda$ be an Artin algebra and $T\in\modu\,(\Lambda)$ be a cotilting $\Lambda$-module. Then, the following statements hold true. 
\begin{itemize}
\item[(a)] $\id(T)=\resdim_{{}^\perp T}(\modu\,(\Lambda)).$
\item[(b)] $\findim(\Lambda)\leq\Phidim(\Lambda)\leq\id(T)+\Phidim({}^\perp T).$
\item[(c)] If $\Phidim({}^\perp T)=0,$ then  $\findim(\Lambda)\leq\Phidim(\Lambda)\leq\Psidim(\Lambda)\leq \id(T).$
\end{itemize}
\end{teo}
\begin{proof} Since $T$ is cotilting, by Lemma \ref{LemaTCH} we know that  $({}^\perp T,(\add\,T)^\wedge)$ is a hereditary complete cotorsion pair in $\modu\,(\Lambda),$  $\add\,T={}^\perp T\cap(\add\,T)^\wedge,$  $({}^\perp T)^\wedge=\modu\,(\Lambda)$ and $\add(T)$ is a ${}^\perp T$-injective and a relative cogenerator in the left saturated class ${}^\perp T.$  Then,  the proof of (b) and (c) can be obtained from Corollary \ref{CotPhidimcota} by taking $n:=0,$ $\omega:=\add(T)$ and $\D:={}^\perp T.$ 
\

Let us prove (a). By the discussed  properties of the cotorsion pair  $({}^\perp T,(\add\,T)^\wedge),$ which are given above, we can apply \cite[Proposition 2.1]{AB}; and thus we get that $\pd_{\add(T)}(M)=\resdim_{{}^\perp T}(M),$ for any $M\in\modu\,(\Lambda).$ Therefore (a) holds true since $\pd_{\add(T)}(\modu\,(\Lambda))=\id(\add(T))=\id(T).$
\end{proof}

As an application of the above theorem, we can get the following result which was firstly obtained in \cite{ GS, LMata}. In order to state this, we recall that an Artin algebra $\Lambda$ is Gorenstein if 
$\id({}_\Lambda\Lambda)$ and $\id(\Lambda_\Lambda)$ are both finite. In this case, it is well known that $n:=\id({}_\Lambda\Lambda)=\id(\Lambda_\Lambda)$ and thus $\Lambda$ is called $n$-Gorenstein. 

\begin{cor}  If $\Lambda$ is a $n$-Gorenstein Artin algebra, then
$$\findim(\Lambda)=\Phidim(\Lambda)=\Psidim(\Lambda)=\mathrm{gl.Gpdim}(\Lambda)=n.$$
\end{cor}
\begin{proof} Let $\Lambda$ be a $n$-Gorenstein Artin algebra. In this case, we have that $T:={}_\Lambda\Lambda$ is a cotilting module in $\modu\,(\Lambda)$ and moreover $\Gproj(\Lambda)={}^\perp T.$ Thus, the corollary follows from Lemma \ref{EjemploBasico} and Theorem \ref{AppCotilting} 
since $\id({}_\Lambda\Lambda)=\id(\Lambda_\Lambda)=\pd(D(\Lambda_\Lambda))\leq \findim(\Lambda),$ where $D:\modu\,(\Lambda)\to\modu\,(\Lambda^{op})$ is the usual duality functor.
\end{proof}

\section{Lat-Igusa-Todorov algebras}

In this section, we examine the definition of Igusa-Todorov algebra, given by J. Wei in \cite{W}, and offer a generalisation in such a way that the finitistic dimension conjecture holds for this new family of algebras and it contains properly the class of Igusa-Todorov algebras.
\ 

The idea for defining an Igusa-Todorov algebra can be seen under the light of an Auslander generator. More explicitly,  for an Artin algebra 
$\Lambda$ which has representation dimension at most $3,$ it is well known that there exists some $V\in\modu\,(\Lambda),$ called Auslander's generator, such that each  $M\in\modu\,(\Lambda)$ admits an exact sequence $0\to V_1\to V_0\to M\to 0$, with $V_0,V_1\in \add\,V.$ In the case of Wei's definition of an Igusa-Todorov algebra \cite{W}, it is requiered the existence of a module $V\in\modu\,(\Lambda)$ and a non negative integer $n$ such that, for every $M\in\modu\,(\Lambda),$ the $n$-syzygy of $M$ admits an exact sequence $0\to V_1\to V_0\to \Omega^nM\to 0$, with $V_0,V_1\in \add\,V.$ The class of Artin algebras satisfying the previous condition are called $n$-Igusa-Todorov algebras ($n$-IT-algebras, for short). Many well known families of algebras are $n$-IT-algebras, for example the classes of: 
\begin{itemize}
\item monomial algebras,
\item special biserial algebras,
\item tilted algebras, 
\item algebras with radical square zero.
\end{itemize} 
Actually, in \cite{W}, J.  Wei proved that the class of $n$-IT-algebras satisfy the finitistic dimension conjecture and raised the question of whether all algebras are Igusa-Todorov algebras. The answer was laying in an earlier paper published by R. Rouquier \cite{R}, as pointed out by T. Conde in \cite{TC}. Using a result on Rouquier's paper, T. Conde was able to provide a family of algebras that are not $n$-IT-algebras, for any $n\geq 0,$ namely the exterior algebra of any vector space that has dimension equal or greater than 3. These examples are instances of self-injective algebras, which have been widely studied and for them being $n$-IT is equivalent to having representation dimension at most $3$. Regarding this last condition there are some classes of self-injective algebras for which it has been proved, those being monomial special multiserial \cite{SS}, wild tilted type \cite{Trepode1} and euclidean type \cite{Trepode2}. The examples, provided by T. Conde, do not belong to any of those classes. Moreover they are of wild representation type.
\

In Section 3, we have introduced the generalised Igusa-Todorov functions and as we have said, one of our main objectives is to generalise the definition given by J. Wei in order to obtain a strictly larger class of algebras that will also satisfy the finitistic dimension conjecture. The rest of the paper is devoted to this purpose.

\begin{defi} An $n$-Lat-Igusa-Todorov algebra ($n$-LIT-algebra, for short), where $n$ is a non-negative integer, is an Artin algebra $\Lambda$ satisfying the following two conditions:
\begin{itemize}
\item[(a)]  there is some class $\D\subseteq\modu\,(\Lambda)$ such that $\add\,\D=\D,$  $\Omega(\D)\subseteq\D$ and $\Phidim\,(\D)=0;$
\item[(b)] there is some $V\in\modu\,(\Lambda)$ satisfying that each $M\in \modu\,(\Lambda)$ admits an exact sequence 
$$0\longrightarrow X_1\longrightarrow X_0\longrightarrow \Omega^nM\longrightarrow 0,$$ such that $X_1=V_1\oplus D_1$, $X_0=V_0\oplus D_0$, with $V_1,V_0\in \add\,V$ and $D_1,D_0\in \D.$
\end{itemize}
In case we need to specify the class $\D$ and the $\Lambda$-module $V,$ in the above definition, we say that $\Lambda$ is a $(n,V,\D)$-LIT-algebra.
\end{defi}

\begin{ex} Let $\Lambda$ be an Artin algebra.
\begin{itemize}
\item[(1)] If $\Lambda$ is an $n$-IT-algebra, then it is also an $n$-LIT-algebra.\\
 Indeed, just take the same module $V,$ appearing in the definition of $n$-IT-algebra, and $\D:=\{0\}.$ 
\item[(2)] If $\Lambda$ is a self-injective Artin algebra, then $\Lambda$ is a $0$-LIT-algebra.\\
 Indeed, it has been proved in \cite{HL} that $\Phidim\,(\Lambda)=0$, so by taking 
$\D:=\modu\,(\Lambda)$ and $V=0,$  we get that $\Lambda$ is a $0$-LIT-algebra.
\end{itemize}   
\end{ex} 

\begin{remark}
In connection with the finitistic dimension conjecture, there is the question of whether the $\Phi$-dimension of a given Artin algebra is always finite. This question has been recently solved independently by M. Barrios, G. Mata \cite{Barrios} and E. Hanson, K. Igusa \cite{Hanson}. By construction, both counterexamples given in \cite{Barrios, Hanson} are Artin algebras of radical cube zero,  and thus by Corollary 3.5 in \cite{W}, they are $1$-IT (also $1$-LIT). In particular, these algebras have finite finitistic dimension.
\end{remark}

As we can see, the class of the $n$-LIT-algebras is strictly larger than the class of the $n$-IT-algebras and contains most of the known examples of algebras satisfying the finitistic dimension conjecture. Next, we prove that the finitistic dimension conjecture holds for the class of $n$-LIT-algebras.

\begin{teo}\label{MainResult}
Let $\Lambda$ be a $(n,V,\D)$-LIT-algebra. Then $$\findim\,(\Lambda)\leq \Psi_{[\D]}(V)+n+1<\infty.$$
\end{teo}
\begin{proof}
Let $M\in\modu\,(\Lambda)$ with $\pd\,M<\infty.$ Since $\Lambda$ is a $(n,V,\D)$-LIT-algebra, there is an exact sequence 
$$0\longrightarrow X_1\longrightarrow X_0\longrightarrow \Omega^nM\longrightarrow 0,$$ such that $X_1=V_1\oplus D_1$, $X_0=V_0\oplus D_0$, with $V_1,V_0\in \add\,V$ and $D_1,D_0\in \D.$
Then, by Theorem \ref{thm_psi}, we get the following inequalities 
$$\pd\,M\leq \pd\,\Omega^nM + n\leq \Psi(X_1\oplus X_0)+1+n.$$
On the other hand,  Proposition \ref{previoLIT} (a) implies that $\Psi(X_1\oplus X_0)\leq \Psi_{[\D]}(X_1\oplus X_0).$ 
Moreover, from Proposition \ref{prop_propiedades} (g) and Proposition \ref{previoLIT} (b), we obtain
$$\Psi_{[\D]}(X_1\oplus X_0)=\Psi_{[\D]}(V_1\oplus V_0\oplus D_0\oplus D_1) = \Psi_{[\D]}(V_1\oplus V_0)\leq\Psi_{[\D]}(V);$$ 
and thus $\pd\,M\leq\Psi_{[\D]}(V)+n+1;$ proving the result.
\end{proof}

As J. Wei did in \cite{W} we also leave the reader with the question of whether all Artin algebras are $(n,V,\D)$-LIT-algebras. To finish we give two propositions and one corollary about these algebras using several results that have been obtained in the previous sections.

\begin{pro}\label{genProp1} Let $\Lambda$ be an Artin algebra satisfying the property: there is some left saturated class $\X$ in $\modu\,(\Lambda)$ with 
$\Phidim\,(\X)=0,$ and a non-negative integer $n$ such that $\resdim_\X(\Omega^n(\modu\,(\Lambda))\leq 1.$ Then,  $\Lambda$ is a $(n,0,\X)$-LIT-algebra and 
$\findim\,(\Lambda)\leq n+1.$
\end{pro}
\begin{proof} Since $\X$ is a left saturated class in $\modu\,(\Lambda),$ it follows that $\add\,\X=\X$ and $\Omega(\X)\subseteq\X.$  Let $M\in\modu\,(\Lambda).$ Using that $\resdim_\X(\Omega^nM)\leq 1,$  we get an 
exact sequence $0\to X'\to X\to \Omega^nM\to 0$ with $X,X'\in\X.$ Thus, we can take $V:=0$ and 
$\D:=\X$ to get that the algebra 
$\Lambda$ is a $n$-LIT-algebra. On the other hand, by Theorem 	\ref{MainResult}, we conclude that 
$\findim(\Lambda)\leq n+1.$
\end{proof}

\begin{cor} Let $\Lambda$ be an Artin algebra satisfying the property:  there is some non-negative integer $n$ such that $\Gpd\,\Omega^n(\modu\,\Lambda)\leq 1.$ Then, $\Lambda$ is a $(n,0,\Gproj(\Lambda))$-LIT-algebra and 
$\findim\,(\Lambda)\leq n+1.$
\end{cor} 
\begin{proof} By taking $\X:=\Gproj(\Lambda)$ in  Proposition \ref{genProp1}, and by applying Lemma \ref{EjemploBasico}, we get that $\Lambda$ is a 
$(n,0,\Gproj(\Lambda))$-LIT-algebra and 
$\findim\,(\Lambda)\leq n+1.$ 
\end{proof}

\begin{pro} Let $\Lambda$ be an Artin algebra and $T\in\modu\,(\Lambda)$ be a cotilting $\Lambda$-module such that 
$\Phidim({}^\perp T)=0$ and $\id\,T\leq 1.$ Then, for any $n\geq 0,$ $\Lambda$ is an $(n, 0,{}^\perp T)$-LIT algebra and 
$\findim(\Lambda)\leq \Phidim(\Lambda)\leq\Psidim(\Lambda)\leq 1.$
\end{pro}
\begin{proof} Since $T$ is cotilting, by Lemma \ref{LemaTCH} we know that  $({}^\perp T,(\add\,T)^\wedge)$ is a hereditary complete cotorsion pair in 
$\modu\,(\Lambda),$  $\add\,T={}^\perp T\cap(\add\,T)^\wedge,$  $({}^\perp T)^\wedge=\modu\,(\Lambda)$ and  
$\add(T)$ is a ${}^\perp T$-injective and a relative cogenerator in the left saturated class ${}^\perp T.$ Moreover by Lemma \ref{previoLIT3(c)}, we have that 
$\resdim_{{}^\perp T}(\Omega^n(\modu\,(\Lambda)))\leq\id(T), \forall n\geq 0.$ Thus, by  Proposition \ref{genProp1}, it follows that $\Lambda$ is an $(n, 0,{}^\perp T)$-LIT algebra, for any $n\geq 0.$ Furthermore, from Theorem \ref{AppCotilting} (c), we get that 
$\findim(\Lambda)\leq \Phidim(\Lambda)\leq\Psidim(\Lambda)\leq \id\,T\leq 1.$ 
\end{proof}

\bibliographystyle{unsrt}

\end{document}